\documentclass[12pt]{amsart}
\usepackage{epsf}
\def\mathbb#1{\mbox{\amss#1}}
\font \amss=msbm10 at 12pt
\usepackage{amsmath}
\usepackage{amssymb}
\usepackage{amsthm}
\newtheorem{teorema}{Theorem}
\newtheorem{predlozhenie}{Proposition}
\newtheorem{lemma}{Lemma}
\newtheorem{remark}{Remark}

\theoremstyle{definition}

\author{E.~Volkov}
\address{Inst. f\"ur Math. Humboldt-Universit\"at zu Berlin, \newline
 Rudower Chaussee 25, 12489 Berlin, Germany}
\email{volkov@math.hu-berlin.de}
\title{A new approach to the family of singularities $Re(x+iy)^m$.}
\date{September 2008}

\begin{document}

\begin{abstract} Assume that $m\ge 2$ and let $l$ be a nonnegative integer with $l\ge m-4$. We give an alternative proof of the fact that any smooth function defined locally around $(0,0)\in \mathbb{R}^2$ with the Taylor power series at $(0,0)$ beginning with $$Re(x+iy)^m+0+...+0$$ ($l$ zeros) is diffeomorphically equivalent to $Re(x+iy)^m$ at $(0,0)$. For $m\ge 5$ and $C\ne 0$ we show that the function $$Re(x+iy)^m+C(x^2+y^2)^{m-2}$$ is not diffeomorphically equivalent to
$Re(x+iy)^m$ at $(0,0)$.
\end{abstract}

\maketitle

\tableofcontents

\section{Introduction.} \label{intr}
\subsection{Setup. Main result. History.} \label{setup}

Assume, we have a class of smooth functions with a common isolated critical point. A normal form for this class is a special member 
of it which can be brought to any other by a smooth coordinate change in some open neighbourhood of the critical point. The word ``special'' usually means 
``polynomial'' and the class of functions is usually specified by some local data at the critical point: the Hessian, the first several terms of the Taylor power series etc. The first example of a normal form is provided by the Lemma 
of Morse: a smooth function with a nondegenerate critical point in a suitable coordinate system looks like a constant term plus an algebraic sum of the squares. The number of minuses in the algebraic sum is determined by the signature of the Hessian. Can we hope for any kind of ``Morse Lemma'' if the 
critical point is degenerate? This was the starting point for the singularity 
theory of smooth functions. 

The following terminology is useful. 
Let $f_1$ and $f_2$ be two real valued functions defined on an open neighbourhood of the point $p$ in $\mathbb{R}^d$. We say that $f_1$ and $f_2$ are (diffeomorphically) equivalent at $p$ and write $f_1\sim f_2$ if $f_1$ can be brought to $f_2$ by a smooth
coordinate change around $p$ fixing $p$. If $f_1$ is equivalent to $f_2$ at $p$, then 
no generality is lost in assuming that $p=(0,...,0)$ and $f_1(p)=f_2(p)=0$. 
All the members of a class that has a normal form are equivalent to each other
(and to the normal form) at the critical point.

Now we turn to the history, which roughly went as follows. 
In late 50th early 60th J.~Nash and J.~Moser
developed a theory which enables one to write clever implicit function 
theorems in the smooth category. 
The theory grew out of the famous Nash embedding
theorem from 1956. In 1968 Samoilenko (cf. \cite{Sam}) used this theory to 
prove a fundamental 
result in singularity theory. Every smooth function near an isolated critical 
point of finite order is equivalent to its 
truncated Taylor 
series at this point. This is a qualitative result: we know the function is
equivalent to a polynomial and to give an upper bound on the degree of the 
polynomial is a separate question. In $1972$ Arnold (cf. \cite{Arnold}) used a 
clever 
``Lie algebraic'' trick, which together with the result by Samoilenko gives 
a very efficient criterion (Lemma 3.2 in \cite{Arnold}) 
to decide that a function is equivalent to a given polynomial.

The setup for this paper goes as follows. For every natural number $m\ge 2$
we consider the class of real-valued smooth functions defined on 
open neighbourhoods of the point $(0,0)$ in $\mathbb{R}^2$ with the leading term of the Taylor power series at $(0,0)$ being equal to $Re(x+iy)^m$. We denote this class by $\Lambda_m$. At first we can hope that $Re(x+iy)^m$ be the normal form for the class $\Lambda_m$. Unfortunately, as we will see later,
it is not true for large $m$. This motivates us to consider smaller classes of functions. Let $n$ be a natural number greater or equal to $m$. 
The class $\Lambda_m^n$ is defined to consist of those functions $f$ from 
$\Lambda_m$, for which the Taylor power series for $f-Re(x+iy)^m$ at $(0,0)$ start from the order $n+1$ or higher. These classes form a nested sequence as $n$ in increases, with the largest class $\Lambda_m^m$ being equal to $\Lambda_m$. In other words,
a member of the class $\Lambda_m^n$ is required to have $n-m$ zeros after the 
leading term $Re(x+iy)^m$ in the Taylor expansion at $(0,0)$. The following theorem
is the main result of this paper.
\begin{teorema} For any $m\ge 2$ we have the following.
\begin{itemize}
\item[(1)] The function $Re(x+iy)^m$ is a normal form for the class 
$\Lambda_m^{s(m)}$ with $s(m):=max(m,2m-4)$. 
\item[(2)] The class $\Lambda_m^{s(m)}$ is the largest in the sequence 
$\{\Lambda_m^n\}_{n\ge m}$ that admits a normal form: for $m\ge 5$ the function 
$$f_{\star}:=Re(x+iy)^m+C(x^2+y^2)^{m-2}\in \Lambda_m^{2m-5}\setminus \Lambda_m^{2m-4},\,\,\,\, C\ne 0$$ 
is not equivalent to $Re(x+iy)^m$ at $(0,0)$.
\end{itemize}
\label{conj}
\end{teorema}
For $m=2,3,4$ we have that $s(m)=m$ and part (2) of Theorem \ref{conj} is 
empty. If $m\ge 5$ $(s(m)=2m-4)$, then should any class $\Lambda_m^n$ 
larger than $\Lambda_m^{2m-4}$ have a normal
form, all the functions from $\Lambda_m^n$ must be equivalent to each other at $(0,0)$. Therefore, showing that the above function $f_{\star}\in \Lambda_m^{2m-5}$
is not equivalent to $Re(x+iy)^m$ does, indeed, give us part (2) of Theorem \ref{conj}.

Part (1) of Theorem \ref{conj} can be (very shortly) proved using Arnold's criterion from \cite{Arnold}. If fact, the above cited work of Arnold contains
the necessary argument for the case $m=3$, see Lemma 5.3 in \cite{Arnold}.
The goal of the paper is to give an independent proof of Theorem \ref{conj} based on completely different ideas. The historical point of the new approach is that we are not using any of the hard analysis of Nash and Moser. Only classical analysis of 19-th century is used: integrability of almost complex structures in dimension $2$ essentially going back to Gauss and a statement that every formal Taylor power series at a point is realized by a certain smooth function defined locally around this point, 
(cf. \cite{Mir}). 

In addition to part (1), our approach will give us part (2) of Theorem \ref{conj} and allow to make the following not very precisely formulated statement rigorous.
\begin{teorema} 
Functions from $\Lambda_m$ equivalent to $Re(x+iy)^m$ form a submanifold of codimension
$(m-2)(m-3)-2$ in $\Lambda_m$.
\label{rigour}
\end{teorema}
Very briefly, some hidden $S^1$-symmetries of the problem --- $Re(x+iy)^m$ solves
the homogeneous Laplace equation for the flat metric on $\mathbb{R}^2$ ---
and some luck with analysis allow us to describe the set of functions from $\Lambda_m$
equivalent to $Re(x+iy)^m$ as a solution set to a certain nonlinear equation
in a finite dimensional space. Theorem \ref{conj} follows at once and 
Theorem \ref{rigour} gets a rigorous meaning without talking about Fr\'echet or Banach manifolds.

\subsection{Sketch of the proof of Theorem \ref{conj}.} 
\label{scet}  
For a function $f\in \Lambda_m$ to be equivalent to $Re(x+iy)^m$ at $(0,0)$
is the same that to be harmonic with respect to some metric $g$ defined locally 
around $(0,0)$ (Section \ref{difeqharm}). We write out the equation 
\begin{equation}
\triangle_gf=0, 
\label{thmain}
\end{equation}
where $\triangle_g$ is the Laplace operator for the metric $g$,
in coordinates $(x,y)$ and view it as a singular 
nonlinear first order partial differential equation for $g$. We do a power series argument, namely we insert
the formal Taylor power series $g_0+g_1+...g_k+...$ for $g$ in 
\eqref{thmain} and see what 
relations on $g_k$ we necessarily have, should the desired metric $g$ exist. 
We arrive at a system of linear equations 
for every degree $k$ (Subsection \ref{PSA1}). 
For degrees starting from $m-3$ (if $m\le 4$,
then for all degrees) the linear system is always solvable (Section \ref{KeyAlg}) 
Only finitely many Taylor power coefficients of $f$ at $(0,0)$
enter in coefficients of the linear systems for low (the only ``problematic'') degrees. This way the success of our power series argument
is reduced to the question whether or not certain nonlinear equation
(Equation \eqref{nonlineq}) on first several
Taylor power coefficients of $f$ is satisfied (Subsection \ref{PSA2}). 
If the power series argument is successful, then we can construct a Riemannian metric $g$ such that $\triangle_gf$ is exponentially small at $(0,0)$
(Section \ref{apprsol}).
If this exponential smallness is the case, then we can modify the Riemannian
metric $g$ in an exponentially small faishon to achieve harmonicity of $f$ with respect
to $g$ in some open neighbourhood around $(0,0)\in \mathbb{R}^2$ (Section \ref{technanal}).
After all, we get that $f$ is harmonic with respect to some $g$ if and only if 
Equation \eqref{nonlineq} is satisfied. 
This is the content of Theorem \ref{central} in Section \ref{peack}.
The rest is simply extracting consequences of Theorem \ref{central}. We prove Theorem \ref{conj} by observing that for any $f\in \Lambda_m^{s(m)}$ Equation \ref{nonlineq}
is trivially satisfied and for $f_{\star}$ $(m\ge 5)$
Equation \ref{nonlineq} is not satisfied. Theorem \ref{rigour} is treated by observing that Equation \ref{nonlineq} defines a submanifold in, roughly speaking, 
the space of Taylor power series truncated at the order $2m-4$ (Theorem \ref{codim}, Section \ref{peack}). 
  
\section{Preliminaries and notation.}
In this section we recall some basic facts and fix notational conventions.
\subsection{The Laplace operator on $\mathbb{R}^2$.}

Let the Riemannian metric $g$ on an open subset $U$ of $\mathbb{R}^2$ 
be defined by the matrix $\{g_{ij}\}_{i,j=1,2}$ 
in the standard coordinates $(x,y)$, that is 
$g_{11}=g(\partial_x,\partial_x), g_{12}=g(\partial_x,\partial_y), 
g_{22}=g(\partial_y,\partial_y)$. The Riemannian metric $g$ on 
$TU$ induces the one on $T^{\star}U$. We denote the 
induced Riemannian metric by the same letter $g$ and consider 
$g^{11}=g(dx,dx), g^{12}=g(dx,dy), g^{22}=g(dy,dy)$. 
It is easily to see that the matrix 
$\{g^{ij}\}^{i,j=1,2}$ is the inverse to the matrix
$\{g_{ij}\}_{i,j=1,2}$. The Hodge-star operator $\star_g$ for the metric $g$ on 
$\Lambda^1T^{\star}U$ is defined as the rotation on $90$ degrees counterclockwise in the fibre over each point. Since rescaling of the metric does not
affect the property of a basis to be orthonormal, the Hodge-star operator does not change if we rescale the metric conformally. It will be a standing 
convention throughout the paper to fix the rescaling 
in such a way that $\det\{g_{ij}\}_{i,j=1,2}=\det\{g^{ij}\}^{i,j=1,2}=1$.   
It is an easy exercise that with the above convention about the determinant, the formulas for the Hodge-star operator 
are $$\star_{g}dx=-g^{12}dx+g^{11}dy$$ and  
$$\star_{g}dy=-g^{22}dx+g^{12}dy,$$ 
so it brings $adx+bdy$ to $(-g^{12}a-g^{22}b)dx+(g^{11}a+g^{12}b)dy,$
i.e. in standard coordinates $(dx,dy)$
the Hodge-star operator is given by the following matrix:
$\left(\begin{array}{lr} 
                            -g^{12}&-g^{22}\\
                             g^{11}&g^{12} \\
                                       \end{array}\right)$.
We define the Laplace operator 
$$\triangle_g:C^{\infty}(U)\longrightarrow \Omega^2(U)$$ 
for the metric $g$ as follows. For a smooth real valued function $f$ on $U$ we set:
$$\triangle_gf:=d\star_gdf.$$
Using the coordinate description of the Hodge-star above it is easy to write a formula
for the Laplacian.
$$\triangle_gf=[(g^{12}f_x+g^{22}f_y)_y+(g^{11}f_x+g^{12}f_y)_x]dx\wedge dy.$$

We make some remarks about notation. Most of the time we work in standard coordinates
$(x,y)$. When convenient, in places were no confusion is possible, we will identify
linear operators (e.g. Hodge star $\star_g$) and bilinear forms (e.g. Riemannian metric $g$) on $\Omega^1(U)$ with the matrices that represent those operators and forms in the basis $(dx,dy)$. We write ``$Id$'' for the identity matrix. We also 
identify $\Omega^{2}(U)$ and $C^{\infty}(U)$ by means of the distinguished volume form
$dx\wedge dy$. This allows us to consider the Laplace operator as taking functions to
functions.   

\subsection{Some notation.}
The space of homogeneous polynomials in $(x,y)$ of degree $n$ with values in 
$\mathbb{K}\in \{\mathbb{R},\mathbb{C}\}$
will be denoted by 
$\mathbb{K}_n[x,y]$; the space of polynomials of degrees not greater than $n$ --- $\mathbb{K}_{\le n}[x,y]$;
of degrees not less than $n$ --- $\mathbb{K}_{n\le }[x,y]$;
of degrees from $n_1$ to $n_2$ --- $\mathbb{K}_{n_1\le n_2}[x,y]$.
Let $f$ be a smooth real valued function defined locally around $(0,0)\in \mathbb{R}^2$. We consider the infinite Taylor power series of $f$ at $(0,0)$  
and different of it. The $n$-th order term of the series
which lives in $\mathbb{R}_n[x,y]$ will be denoted by $[f]_n$. The following notation
is self explanatory: $$[f]_{n\le}\in \mathbb{R}_{n\le }[x,y],\,\,\, 
[f]_{\le n}\in \mathbb{R}_{\le n}[x,y],\,\,\, 
[f]_{n_1\le n_2}\in \mathbb{R}_{n_1\le n_2}[x,y].$$ If $f\in \Lambda_m$,
then $[f]_m=Re(x+iy)^m$. In this case we will write $f_0$ for $[f]_m$. This will
save some typing when writing partial derivatives of $f_0$ as $f_{0x}$ and 
$f_{0y}$. 

\subsection{Homogeneous polynomials.} 
\label{homog}

In this subsection we recall some elementary facts about homogeneous polynomials
in two variables.  
Let $P$ be a homogeneous 
polynomial in $(x,y)$ of the $n$-th degree ($\deg P=n$). Then there is a positive 
constant $C$ such that $$|P(x,y)|\le C (x^2+y^2)^{\frac{n}{2}}$$
for all $(x,y)\in \mathbb{R}^2$ . The zero set
$$KerP=\{(x,y)\in \mathbb{R}^2| P(x,y)=0\}$$ of the polynomial $P$ is
either a 1-point set $\{(0,0)\}$ or a finite union of lines through the 
origin (1-dimensional linear subspaces)
(we leave off the trivial case of the zero polynomial). If 
$KerP=\{(0,0)\}$, then there are positive constants $c,C$ such that
\begin{equation}
c(x^2+y^2)^{\frac{n}{2}}\le |P(x,y)| \le C(x^2+y^2)^{\frac{n}{2}}.
\label{poly1}
\end{equation}
Assume now that $KerP=\cup_{i\in I}l^1_i$, where
$l^1_i\subset \mathbb{R}^2$ 
is a linear subspace of dimension $1$ and $I$ is finite.
We take a small positive $\delta$ and set 
$$Cone_i^{\delta}(P)=\{(\xi,\eta)\in \mathbb{R}^2\setminus \{(0,0)\}| 
\frac{dist((\xi,\eta),l^1_i)}{dist((\xi,\eta),(0,0))}< \delta\} $$
and $$\Omega^{\delta}(P)=\mathbb{R}^2\setminus \cup_{i\in I}Cone_i.$$
Then there are positive constants $c_{\delta}$ and $C_{\delta}$
such that 
\begin{equation}
c_{\delta}(x^2+y^2)^{\frac{n}{2}}\le |P|_{\Omega^{\delta}(P)}(x,y)| 
\le C_{\delta}(x^2+y^2)^{\frac{n}{2}}.
\label{poly2}
\end{equation}

\subsection{Exponentially small functions.} 
\label{O(exp)}

In this section we recall some facts about the ideal of functions 
exponentially small near a point.
We denote the set of infinitely differentiable 
functions defined in some open neighbourhood of $(0,0)\in \mathbb{R}^2$ 
with vanishing Taylor series at $(0,0)\in \mathbb{R}^2$ by $O(exp)$. This 
is an ideal in the ring of all (locally defined) smooth functions, i.e. a 
multiplication of an element from $O(exp)$ with any smooth function
gives us again a function of the class $O(exp)$. This class respects the 
operation of taking derivatives, i.e. partial derivatives of all orders taken
from an $O(exp)$-function belong to $O(exp)$ (algebraically, $O(exp)$ is 
a differential ideal). The following criterion is a standard way to check
that a given function belongs to the class $O(exp)$. 
\begin{lemma}
Let the function $\phi$
be smooth in a punctured neighbourhood of $(0,0)$. If $\phi$ decays at 
$(0,0)$ together with all its derivatives faster than any polynomial,
then the continuation of $\phi$ across the origin by $0$ belongs to
the class $O(exp)$. 
\end{lemma}
\begin{proof}
It suffices to prove that all partial derivatives of $\phi$ at $(0,0)$ exist
and vanish. Take for instance the first partial derivative with respect to
$x$, namely $$\partial_x\phi_{(0,0)}:=
\lim_{h\to 0}\frac{\phi(h,0)-\phi(0,0)}{h}=
\lim_{h\to 0}\frac{\phi(h,0)}{h}.$$ The last limit exists and vanishes, since
$\phi$ decays at $(0,0)$ faster than any polynomial. Similarly, 
$\partial_y\phi_{(0,0)}=0$. For the partial derivatives of the second order
similar procedure works. It uses that the partial derivatives of $\phi$
of the first order decay faster than any polynomial. Proceeding inductively
one shows that all partial derivatives of $\phi$ at $(0,0)$ exist and 
vanish. \end{proof}
This has an immediate application. 
\begin{lemma}
Let the function 
$\phi$ be of the class $O(exp)$ and $\sigma$ be a (locally defined) smooth 
function which has zero of finite order at $(0,0)$ and does not vanish in 
a punctured neighbourhood of $(0,0)$. Then the ratio $\frac{\phi}{\sigma}$ is
well defined locally around the origin and belongs to the class $O(exp)$.
\end{lemma}
\begin{proof}
Writing out a partial derivative of some order of the fraction  
$\frac{\phi}{\sigma}$ gives us a fraction whose numerator decays faster than 
any polynomial and the denominator is equal to $\sigma^n$ for some natural $n$.
The function $\sigma$ having isolated zero of finite order at $(0,0)$ implies
that its leading term $P$ in Taylor expansion at $(0,0)$ has the unique zero
at $(0,0)$. Therefore (see Inequality \eqref{poly1}) $P$
and hence $\sigma^n$ can be estimated from below
by $c(x^2+y^2)^{n/2}$ for some positive constant $c$. 
Altogether, the fraction, 
representing the partial derivative decays faster than any polynomial at 
$(0,0)$. Application of the previous lemma to $\frac{\phi}{\sigma}$
finishes the proof.\end{proof}
There is one more technical remark that we will need in future. Let 
$\phi$ be of the class $O(exp)$ and $\sigma$ be a (locally defined) smooth 
function which has a zero of finite order at $(0,0)$. 
Let the homogeneous polynomial 
$P$ be the leading term in the Taylor series of $\sigma$ at $(0,0)$. 
Assume that $KerP=\cup_{i\in I}l^1_i$, where
$l^1_i\subset \mathbb{R}^2$ 
is a linear subspace of dimension $1$ with finite $I$ and take a small positive
$\delta$ to define $\Omega^{\delta}(P)$ as in the previous section.
Then $\frac{\phi}{\sigma}|_{\Omega^{\delta}(P)}$ decays at $(0,0)$ together 
with all its derivatives faster than any polynomial 
(Inequality \eqref{poly2} is used for the proof).

These facts about
the class $O(exp)$ will be used freely later on without special references.
In calculations, by abuse of notation, we will sometimes denote an 
$O(exp)$-function by the symbol $O(exp)$. The $O(exp)$ notation for $2$-forms
defined in a neighbourhood of $(0,0)\in \mathbb{R}^2$ transfers by means
of a fixed volume form.

\section{Diffeomorphic equivalence to $Re(x+iy)^m$ and harmonicity.}
\label{difeqharm}
In this section we start approaching Theorem \ref{conj}.
The real $2$-plain $\mathbb{R}^2$ is given the standard coordinates $(x,y)$. 
\begin{predlozhenie}
Let $f$ be a function from $\Lambda_m$. 
Then the following two assertions are equivalent:\\
$1)$ The function $f$ is equivalent to $Re(x+iy)^m$ at $(0,0)$.\\
$2)$ There exists a Riemannian metric $g$ defined on some open neighbourhood of $(0,0)$
making $f$ harmonic.\\
Moreover, if $1)$ and $2)$ hold, then we can assume that in $1)$ the leading term at $(0,0)$ of the diffeomorphism which realizes the equivalence is equal to identity
and in $2)$ the Riemannian metric $g$ at $(0,0)$ is the standard Euclidean.
\label{harmreform}
\end{predlozhenie}
\begin{proof} The implication from 1) to 2) is obvious, since the function
$Re(x+iy)^m$ is harmonic with respect to the standard (Euclidean) metric on 
$\mathbb{R}^2$. For the converse, take a Riemannian metric $g$ which makes $f$
harmonic and consider an almost complex structure $J$ 
induced by $g$ --- a rotation by $90$ degrees 
counterclockwise. We can assume a small ball $B$ around $(0,0)$
to be the common domain of definition of $f$ and $g$. 
For dimension reasons any almost complex structure on $B$
is integrable. This means that assuming the ball $B$ to be small enough we can pick
a complex coordinate $z=u+iv$ such that $J\partial_u=\partial_v$,
i.e. $\partial_v$ is obtained from $\partial_u$ by means of the rotation by
$90$ degrees (counterclockwise). This implies that $g$ looks like a multiple
of identity in the coordinate system $(u,v)$. (Existence of such a coordinate system 
(``isothermal coordinates'') essentially goes back to Gauss.)
Since for dimension reasons the Hodge-star operator, and therefore the Laplace operator on $B$ does not change if we re-scale the metric conformally, 
we have that $f$ is harmonic with respect to the metric
represented by the identity matrix
in the coordinates $(u,v)$. This means that $f$ in coordinates $(u,v)$
is a real part of some complex-valued function $F$ 
which depends holomorphically on $z=u+iv$.
The last step is to bring 
$F=a_mz^m+a_{m+1}z^{m+1}+...$, $a_j\in \mathbb{C}$
to its leading power $z^m$ by a biholomorphic change of coordinates. So finally
$f=ReF=Rez^m=Re(u+iv)^m$. Moreover, we can assume that all our coordinate 
changes preserved the origin.
This completes the proof of equivalence of $1)$ and $2)$.
The function $f$ has the same leading term at $(0,0)$ as $Re(x+iy)^m$.
From this it is easy to see that if $1)$ holds, then the diffeomorphism
which realizes the equivalence can be taken to have identity as the leading term
at $(0,0)$. This, in turn, gives us that the Riemannian metric $g$ in $2)$ can be assumed to satisfy $g_{(0,0)}=Id$.   
\end{proof}
This proposition reduces the problem of diffeomorphic equivalence of $f\in \Lambda_m$ and $Re(x+iy)^m$ to the inverse problem for the Laplace operator. We start dealing with the inverse problem in the next section. When looking for the desired 
Rieannian metric $g$ or considering a metric which solves our problem we will,
by convention, always assume that $g_{(0,0)}=Id$. This is no loss of generality by 
Proposition \ref{harmreform}.

\section{Power series argument.}
\label{Psa}

In this section we insert a formal Taylor power series for $g$ in 
$$\triangle_gf=0$$
and do a power series argument. The notation will be as follows. We write the infinite
formal Taylor power series for $g$ as 
$$g=g_0+g_1+...+g_k+... $$
(with $g_0=Id$ as we agreed).
We do everything in standard coordinates $(x,y)$, 
so $g_k$ is viewed as a symmetric $2\times 2$ 
matrix with entries in $\mathbb{R}_k[x,y]$.
The following notation is motivated by how we produce the matrix for the Hodge-star
operator out of the matrix for the metric. Given a $2\times 2$ symmetric matrix 
$$g_k=\left(\begin{array}{lr} 
                            G^{11}&G^{12}\\
                            G^{12}&G^{22} \\
                                       \end{array}\right),$$
we introduce a traceless $2\times 2$ matrix                                      
$$G_k:=\left(\begin{array}{lr} 
                            -G^{12}&-G^{22}\\
                             G^{11}&G^{12} \\
                                       \end{array}\right).$$
We set $T_k=G_0+...+G_k$. 
First, we assume that there exists a locally defined Riemannian metric $g$, such that 
$\triangle_gf$ vanishes at $(0,0)$ up to infinite order and see what relations it gives for Taylor power coefficients $g_k$ of $g$ and then we analyse these relations algebraically.
\subsection{Inductive setup.}
\label{PSA1}
\begin{predlozhenie}
Let $f\in \Lambda_m$ and let $g$ be a Riemannian metric defined locally around $(0,0)$
such that 
\begin{equation}
\triangle_gf=O(exp),
\label{apprharm} 
\end{equation}
then the following assertion $\mathcal{A}_k$ holds true for all positive integers $k$:
$$(i) \,\,  [dT_kdf]_n=0,\,\, n=0,1,...,k+m-2 $$ and
$$(ii)\,\,  [\det T_k]_n=0,\,\,  n=1,...,k. $$
\label{power}
\end{predlozhenie}
\begin{proof}
We begin with $(ii)$. Recall that we agreed to normalize $g$ such that $\det g=\det\star_{g}=1$. Fix a positive integer $k$ and a positive integer $n$ with
$n\le k$. Taylor coefficients of $\star_{g}$ of orders $k+1$ and higher do not contribute to the $n$-th Taylor coefficient of $\det\star_{g}$.
Therefore, $0=[det\star_g]_n=[det[\star_g]_k]_n=[detT_k]_n$. This gives us $(ii)$.
We use Equation \eqref{apprharm}: fix a positive integer $k$ and leave a positive integer $n$ vary. Then, $0=[d\star_gdf]_n=[dT_kdf]_n+[d[\star_g]_{\ge{k+1}}df]_n$.
For $n=0,...,k+m-2$ the second term vanishes and we get $(i)$.
\end{proof}

Now we fix some $f\in \Lambda_m$, step away from the actual metric $g$ and work with formal power series 
$g_0+g_1+...$.
Note that for $k=0$ the assertion $\mathcal{A}_k$ is always true. 
The next goal is to assume the assertion to
be true for $k-1$ and rewrite equivalently the assertion for $k$ under this assumption.
Clearly, $$dT_kdf=d(T_{k-1}+G_k)df=dT_{k-1}df+dG_kdf.$$
Consider $n=0,1,...,k+m-3=(k-1)+m-2$. 
Since the leading term of $dG_kdf$ is of order $k+m-2$ 
we have that $[dG_kdf]_n=0$. 
By our assumption $[dT_{k-1}df]_n=0$.
It means that for the above values of $n$ Equation (i) is satisfied.
It remains to look at Equation (i) for $n=k+m-2$. That is we are looking at 
$$[dT_kdf]_{k+m-2}=0$$ which is the same as
$$[dG_kdf]_{k+m-2}=-[dT_{k-1}df]_{k+m-2}.$$
Since the leading power of $f-f_0$ is greater than $m$ 
the last equation is equivalent to    
\begin{equation}
[dG_kdf_0]_{k+m-2}=-[dT_{k-1}df]_{k+m-2}.
\label{basic}
\end{equation}
Now we consider 
Equation (ii). Since it is automatically true for $n=1,...,k-1$ by our 
assumption, the equation will follow if 
\begin{equation}
[\det T_k]_k=0.
\label{determ}
\end{equation}
Clearly, $[\det T_k]_k=G^{11}_k+G^{22}_k+
[\det T_{k-1}]_k$, therefore 
Equation \eqref{determ} is equivalent to  
\begin{equation}
G^{22}_k=[\det T_{k-1}]_k-G^{11}_k.
\label{G22}
\end{equation}
Now we substitute this in Equation \eqref{basic}, 
obtaining the following equation:
\begin{equation}
(G^{12}_kf_{0x}-G^{11}_kf_{0y})_y+(G^{11}_kf_{0x}+G^{12}_kf_{0y})_x=\phi_kdx\wedge dy,
\label{subst}
\end{equation}
where 
\begin{equation}
\phi_k=-[dT_{k-1}df]_{k+m-2}/dx\wedge dy-([\det T_{k-1}]_kf_{0y})_y. 
\label{rhs}
\end{equation}
The right hand side $\phi_k$ of Equation \eqref{subst} depends only on the terms of the power series 
$g_0+g_1+...$ up to order $k-1$ and the left hand side involves $G^{11}_k$ and $G^{12}_k$. 
So we can look at \eqref{subst} as an equation for $(G^{11}_k,G^{12}_k)$.

Now we 
summarize our discussion of the assertion $\mathcal{A}_k$ as follows. Fix any $k\ge 1$.
Then under the assumption that $\mathcal{A}_{k-1}$ is true, we have that $\mathcal{A}_k$ is equivalent to the following: 
\begin{itemize}
 \item $\phi_k$ is such that Equation \eqref{subst} is solvable, 
 \item the pair $(G^{11}_k,G^{12}_k)$ solves Equation \eqref{subst} and
 \item $G^{22}_k$ meets Equation \eqref{G22}. 
\end{itemize}
Assume now that given $f\in \Lambda_m$
we want to construct a power series $g_0+g_1+...$ such that $\mathcal{A}_k$
holds for all $k$. We set $g_0=Id$ and this gives us that $\mathcal{A}_0$ is true
and allows to compute $\phi_1$ (using \eqref{rhs}). 
Then we look at Equation \eqref{subst} for $k=1$. If it is 
not solvable, then we stop, else we solve it to get $G^{11}_1$
and $G^{12}_1$. Then we use \eqref{G22} to get $G^{22}_1$. Having $g_1$ at hands we can
use \eqref{rhs} to compute $\phi_2$. We look at \eqref{subst} for $k=2$ and proceed
inductively to construct successively the whole infinite power series $g_0+g_1+...$ 
or to get stuck at a certain point. 
Note that the check on 
the step $k$ happens conditionally to the fact that we were successful on the step $k-1$: Equation \eqref{rhs} for the right hand side of \eqref{subst} involves $T_{k-1}$, implicitly assuming that Equation \eqref{subst} for $k-1$ was solved.
This subtlety will be discussed in more detail in the next subsection.

The following notation will prove useful when rewriting the left hand side of \eqref{subst}
in a more convenient way. Let the linear operator 
$$L:\mathbb{R}^2\longrightarrow Mat(2\times 2,\mathbb{R})$$
be defined as follows

\begin{equation}
L(a,b):=\left(\begin{array}{cc}
                                    -b&a\\
                                     a&b\\
                       \end{array}\right).
\label{complex}
\end{equation}                       
With this notation we define the linear operator                        
$$\Theta_k:\mathbb{R}_k[x,y]\times \mathbb{R}_k[x,y]\longrightarrow \mathbb{R}_{k+m-2}[x,y]$$
via $$\Theta_k(Q_1,Q_2):=d(L(Q_1,Q_2)df_0)/dx\wedge dy.$$
With this Equation \eqref{subst} rewrites as 
\begin{equation}
\Theta_k(G^{11}_k,G^{12}_k)=\phi_k.
\label{rewrsubst}
\end{equation}
As we will show in Section \ref{KeyAlg}, the operator $\Theta_k$ has maximal rank for all $k\ge 1$.
We take this for granted at the moment and look at consequences. 

Counting dimensions of
the domain and the target spaces of the operator $\Theta_k$ shows that for $k\ge m-3$,
Equation \eqref{rewrsubst} is always solvable. Thus, if we managed to 
make the steps for $k=1,...,m-4$, then we can {\it always} make the remaining steps for 
$k=m-3,m-2,...$, i.e. construct an infinite formal Taylor power series $g_0+g_1+...$
such that assertion $\mathcal{A}_k$ holds true for all $k\ge 1$. 
We close this subsection by noting that for $k=1,....,m-4$ solvability of Equation 
\eqref{rewrsubst} crucially depends $\phi_k$. This forces us to study $\phi_k$ 
a little deeper. We make an easy remark about $\phi_k$ now and leave further discussion for the next subsection.
\begin{remark}
Fix a positive integer $N\ge m$, assume that $[f-f_0]_n=0$ for $m<n<N$. Then $\phi_k=0$ for $1\le k<N-m$ and 
$\phi_k=-\triangle_{st}([f-f_0]_N)$ for $k=N-m$ 
\label{comprhs}
\end{remark}

\subsection{The function $\phi$.}
\label{PSA2}
Assume that $m\ge 5$. Otherwise Equation \eqref{rewrsubst} is always solvable and 
the following discussion becomes empty. We consider the set $S:=\{1\le k\le m-4\}$
of those values of $k$ for which solvability of Equation \eqref{rewrsubst} is problematic. 
The first goal of this subsection 
is to make sense of $\phi_k$ regardless of the solvability of Equation \eqref{rewrsubst}
for $k-1$.  We view Equation \eqref{rewrsubst}
as a problem of finding the preimage of the element 
$\phi_k\in \mathbb{R}_{k+m-2}[x,y]$ under the linear map
$\Theta_k$. We introduce an auxiliary Euclidean structure on the target space $\mathbb{R}_{k+m-2}$ of $\Theta_k$
and let $P_{\Theta_k}$ be the orthogonal projection to the image of $\Theta_k$.
We replace \eqref{rewrsubst} by the following Equation 
$$\Theta_k(G^{11}_k,G^{12}_k)=P_{\Theta_k}\phi_k$$
and set 
\begin{equation}
(G^{11}_k,G^{12}_k)=\Theta_k^{-1}P_{\Theta_k}\phi_{k-1}
\label{indG}
\end{equation}
and then use it to build 
\begin{equation}
T_k=T_{k-1}+G_k.
\label{indT} 
\end{equation}
And then we use this 
$T_k$ to write the right hand side for Equation \eqref{rewrsubst} for $k+1$ by Formula
\eqref{rhs} above. This way we get a sequence $\phi_1,...,\phi_{m-4}$.
At this point we introduce an auxiliary Euclidean
structure on $\mathbb{R}_k[x,y]\times \mathbb{R}_k[x,y]$. Together with the one on
$\mathbb{R}_{k+m-2}$, this allows us to consider the adjoint operator to $\Theta_k$:
$$\Theta^{\star}_k:\mathbb{R}_{k+m-2}[x,y]\longrightarrow \mathbb{R}_k[x,y]\times\mathbb{R}_k[x,y]$$
and solvability of the set of Equations \eqref{rewrsubst}, $k\in S$ reads as 
$$\Theta^{\star}_k\phi_k=0,\,\,\, k\in S.$$
Now we let $f\in \Lambda_m$ vary, so
the sequence $\{\phi_k\}_{k\in S}$ becomes a function of $(m-4)$ Taylor coefficients of $f$ at $(0,0)$ --- the ones of orders from $m+1$ to $m+(m-4)$ (the leading term $Re(x+iy)^m$ is fixed throughout). 
We view the sequence of functions $\{\phi_k\}_{k\in S}$
as one function 
$$\phi:\mathbb{R}_{m+1\le 2m-4}[x,y]\longrightarrow \mathbb{R}_{m-1\le 2m-6}[x,y].$$
and the sequence of operators $\{\Theta_k\}_{k\in S}$ as one operator 
$$\Theta:\mathbb{R}_{1\le m-4}[x,y]\times \mathbb{R}_{1\le m-4}[x,y]\longrightarrow \mathbb{R}_{m-1\le 2m-6}[x,y].$$

The newly developed terminology allows us to summarize our discussion of $\mathcal{A}_k$ as follows. Given the Taylor polynomial
$[f]_{1\le 2m-4}$ of $f$ we use Formulas \eqref{G22}, \eqref{rhs}, \eqref{indG} and \eqref{indT}
to construct the sequence $\phi=\{\phi_k\}_{k\in S}$ inductively. 
The relation between assertions $\{\mathcal{A}_k\}_{k\in S}$ and the function $\phi$
is expressed by the following 
\begin{predlozhenie} Let $f\in \Lambda_m$ and set $h:=[f]_{1\le m+(m-4)}$. Then
the following are equivalent:
\begin{itemize}
\item [(i)]  There exist a formal Taylor power series $g_0+g_1+...$
such that assertion $\mathcal{A}_k$ is true for all $k\ge 1$.\\
\item [(ii)] The following equation
     \begin{equation}
     \Theta^{\star}\phi(h)=0
     \label{nonlineq}
     \end{equation}  
holds true.
\end{itemize}
\label{endpsr}
\end{predlozhenie}
For $m=2,3,4$ we have that $m+1>2m-4$ and $h\in \mathbb{R}_{m+1\le 2m-4}[x,y]$ is simply not defined. For these values of $m$ we agree to say that Equation \eqref{nonlineq}
holds true (by default).

The second goal of this subsection is to get more information about  
$\phi$ as a function of $h=[f]_{m+1\le m+(m-4)}$. Note that something we know
already, namely, $\phi(0)=0$ and see also Remark \ref{comprhs}. The following
proposition gives us more hand on $\phi$. 
\begin{predlozhenie}
The function $\phi$ is a submersion. That is the derivative of it is surjective at 
every point.
\label{reg}
\end{predlozhenie} 
\begin{proof}
The idea is to fix a point $h\in \mathbb{R}_{m+1\le 2m-4}[x,y]$
and and write out the derivative of the function $\phi$ at $h$ with respect to 
certain splittings of the domain and of the target. Namely, 
 
 $$\mathbb{R}_{m+1\le 2m-4}[x,y]=\bigoplus_{l=1}^{m-4} \mathbb{R}_{m+l}[x,y],$$
 $$\mathbb{R}_{m-1\le 2m-6}[x,y]=\bigoplus_{k=1}^{m-4} \mathbb{R}_{m+k-2}[x,y].$$
We decompose $h$ as $h=\Sigma_{l=1}^{m-4}h_l$ with $h_l\in \mathbb{R}_{m+l}[x,y]$.
Note that since by construction $\phi_k$ depends only on $h_1,...,h_k$, the derivative
of $\phi$ at $h$ written out with respect to the splittings above has an upper-triangular form with the operator $\partial{\phi_k}/\partial{h_k}$ on the diagonal, so it remains to show that the last operator is surjective. 
Differentiation of Equation \ref{rhs} with respect $h_k$ gives us
  $$\frac{\partial\phi_k}{\partial h_k}=
  -[dT_{k-1}d\frac{\partial{f}}{\partial{h_k}}]_{k+m-2}/dx\wedge dy.$$
The other terms vanish, because $T_{k-1}$ does not depend on $h_k$. Since 
$h_k\in \mathbb{R}_{m+k}[x,y]$, we can rewrite 
  $\partial\phi_k/\partial h_k$ as $-dT_0d$ which is up to a sign the standard Laplacian 
  $$\triangle_{st}:\mathbb{R}_{m+k}[x,y]\longrightarrow \mathbb{R}_{m+k-2}[x,y],$$
and the latter is a surjective linear operator between linear spaces as written above.
\end{proof}
This proposition allows us to establish a lemma that will help us to deal with Theorem \ref{rigour}.
\begin{lemma}
The set $$\mathcal{E}:=\{h\in \mathbb{R}_{m+1\le 2m-4}[x,y]|\Theta^{\star}\phi(h)=0\}$$
is a submanifold in $\mathbb{R}_{m+1\le 2m-4}[x,y]$ of codimension $(m-2)(m-3)-2$.
\label{submers}
\end{lemma}
\begin{proof}
Recall that the oprator $\Theta$ has maximal rank. Counting dimensions for the domain and the target spaces of $\Theta$ shows that it is injective. 
Therefore its adjoint $\Theta^{\star}$ is surjective. Thus,
the derivative of the composition $\Theta^{\star}\circ \phi$ is surjective at every point, so the level sets of $\Theta^{\star}\circ \phi$ are submanifolds. The codimension of a level set is
the dimension of the target space of $\Theta^{\star}\circ \phi$, that is
$$\Sigma_{k=1}^{m-4}dim(\mathbb{R}_k[x,y]\times \mathbb{R}_k[x,y])=(m-3)(m-2)-2.$$
\end{proof}

\section{Key algebraic trick.}
\label{KeyAlg}
In this section we analyse algebraic properties of the operator $\Theta_k$. 
The main goal is the following 
\begin{predlozhenie}
For any $k\ge 1$ the operator $\Theta_k$ has maximal rank.
\label{maxrank}
\end{predlozhenie}
Prior to giving the proof we rewrite 
$$\Theta_k:\mathbb{C}_k[x,y]\cong \mathbb{R}_k[x,y]\times \mathbb{R}_k[x,y]\longrightarrow
\mathbb{R}_{k+m-2}[x,y]$$
and bring some representation
theory into play. Formula \eqref{complex} for the map $L$ suggests to introduce complex notation:
$z=x+iy$, $\partial_z=\frac{1}{2}(\partial_x-i\partial_y)$ and $Q:=Q_1+iQ_2\in \mathbb{C}_k[x,y]$
for $(Q_1,Q_2)\in \mathbb{R}_k[x,y]\times \mathbb{R}_k[x,y]$.

It is an elementary calculation 
that $$\Theta_k(Q)=2mRe\partial_z(Qz^{m-1}).$$
We recall the standard action of the group $S^1$ (considered as the unit circle in $\mathbb{C}$) on $\mathbb{R}_{k+m-2}[x,y]$
(the target space of $\Theta_k$) by coordinate change: for $s\in S^1$ and  $P\in\mathbb{R}_{k+m-2}[x,y]$ we set $$(sP)(x,y):=P(s(x+iy)).$$
We decompose the target space of $\Theta_k$ into the irreducible summands of this representation:
$$\mathbb{R}_{k+m-2}[x,y]=\Sigma_{q=0}^{[(k+m-2)/2]}Irr_{k+m-2}^q,$$
where $$Irr_{k+m-2}^q= (x^2+y^2)^qSpan_{\mathbb{R}}\{Re(x+iy)^p,Im(x+iy)^p\},$$ 
$p+2q=k+m-2$. Here ``$[\cdot]$'' denotes ``the largest integer not greater than''. Now we can start the proof of Proposition \ref{maxrank}. 
\begin{proof}
The operator $\Theta_k$ having maximal rank is equivalent to
the following inequality 
$$\dim Im\Theta_k\ge min(\dim \mathbb{C}_k[x,y], \dim \mathbb{R}_{k+m-2}[x,y])=min(2(k+1),k+m-1).$$
We set $$M(k):=min(k,[(k+m-2)/2])$$ and show that 
\begin{equation}
\Sigma_{q=0}^{M(k)}Irr_{k+m-2}^q\subset Im\Theta_k. 
\label{inclim}
\end{equation}
Indeed, for a nonnegative $q\le M(k)$ we
set $n:=k-q$ and note that $p=k+m-2-2q=n+m-2-q$.
We compute:
$$\frac{1}{2m}\Theta_k(\bar{z}^qz^n)=Re(\bar{z}^qz^nz^{m-1})_z=
Re(z^{n+m-1}\bar{z}^q)_z=$$ 
$$=Re(n+m-1)z^{n+m-2-q+q}\bar{z}^q=Re(n+m-1)z^pz^q\bar{z}^q=$$
$$=(n+m-1)(x^2+y^2)^qRe(x+iy)^p,$$
and analogously $$\frac{1}{2m}\Theta_k(i\bar{z}^qz^n)=-(n+m-1)(x^2+y^2)^qIm(x+iy)^q,$$
meaning that $Irr_{k+m-2}^q\subset Im\Theta_k$. Now Inclusion 
\eqref{inclim} implies the inequality on the level of 
dimensions: $$\dim Im\Theta_k\ge \dim (\Sigma_{q=0}^{M(k)}Irr_{k+m-2}^q)=min(2(k+1),k+m-1).$$
\end{proof}
The primary goal of this section has been achieved, but we want make an 
additional remark, 
concerning the case when the operator $\Theta_k$ is injective, but not surjective:
we point out explicitly a subspace that is missed by the image of $\Theta_k$. 
\begin{remark}
Assume that $m\ge 5$. Then $(x^2+y^2)^{m-3}\notin Im\Theta_{m-4}$  
\label{missed}
\end{remark}
\begin{proof}
In this case the dimension of the domain equals $2m-6$ which is strictly less than 
$(2m-5)$ --- the dimension of the target space. So, we know there is a $1$-dimensional
subspace of the target which misses the image of the operator. More precisely,
the target space decomposes as follows:
$$\mathbb{R}_{2m-6}[x,y]=\Sigma_{q=0}^{m-4}Irr_{2m-6}^q\oplus Irr_{2m-6}^{m-3},$$
as we showed in the proof of Proposition \ref{maxrank} (see Inclusion \eqref{inclim}), 
the first summand belongs to the image of $\Theta_k$,
but the dimension of the first summand equals the dimension of the domain of $\Theta_k$,
therefore the first summand is exactly the whole image of $\Theta_k$, so the second must be missed. It writes out as $Span_{\mathbb{R}}{(x^2+y^2)^{m-3}}$.
\end{proof}

\section{Approximate solution.}
\label{apprsol}
We keep all the notation from Section \ref{Psa}.

\begin{predlozhenie}
Let $f\in \Lambda_m$ and set $h:=[f]_{m+1\le 2m-4}$. Then the following are equivalent.
\begin{itemize}
\item[(1)]  There exists a locally defined Riemannian metric $g$ such that 
$$\triangle_gf=O(exp).$$
\item[(2)]  Equation \eqref{nonlineq} holds true. 
\end{itemize}

\label{approx}
\end{predlozhenie}
\begin{proof}
The implication from $(1)$ to $(2)$ is a combination of Proposition \ref{power} and 
Proposition \ref{endpsr}. Assume, conversely, that $(2)$ holds. This gives us an infinite
power series $g_0+g_1+...+g_k+...$ or equivalently $G_0+G_1+...+G_k+...$
such that $\mathcal{A}_k$ is true for all $k$. By the result of Mirkil \cite{Mir}
there exits a $2\times 2$ matrix $\tilde T$ whose entries are smooth functions defined locally around $(0,0)$ such that Taylor power series of $\tilde T$ at $(0,0)$ is exactly $G_0+G_1+...+G_k+...$. We apply Proposition \ref{endpsr} and use assertions 
$\mathcal{A}_k$ to get the following asymptotic behaviour of $\tilde T$ at $(0,0)$:\\
(a) $d\tilde Tdf=O(exp)$,\\
(b) $Trace \tilde T=O(exp)$,\\ 
(c) $\det \tilde T=1+O(exp)$.\\
Now we are going to correct the matrix $\tilde T$ a little bit.
The property (b) says that the sum of the diagonal elements of the operator
$\tilde T$ is of the class $O(exp)$, therefore by changing the lower right 
element of $\tilde T$, we can achieve that the new operator 
(called again $\tilde T$) is now traceless and the properties (a) and (c) 
still hold true. The last step is to set $T=(\det \tilde T)^{-1/2}\tilde T$
(we shrink the domain of definition of $\tilde T$ if necessary to ensure that 
$\det\tilde T>0$). 
Clearly, the matrix $T$ is traceless and $\det T=1$. We check whether it
satisfies property (a). Indeed, $$dTdf=d ((det \tilde T)^{-1/2}\tilde Tdf)
=d(det \tilde T)^{-1/2}\wedge \tilde Tdf+(det \tilde T)^{-1/2}d\tilde Tdf.$$
Working out the two terms one by one gives us:
$$d(det \tilde T)^{-1/2}\wedge \tilde Tdf=d(1+O(exp))^{-1/2}\wedge 
\tilde Tdf=$$
$$=d(1+O(exp))\wedge \tilde Tdf=d(O(exp))\wedge \tilde Tdf=
O(exp);$$ $$(det \tilde T)^{-1/2}d\tilde Tdf=(1+O(exp))O(exp)=O(exp).$$

Since $T$ is traceless, it can be written as $\left(\begin{array}{cc}
                               -g^{12}&-g^{22}\\
                                g^{11}&g^{12}\\
                       \end{array}\right)$.
We set the Riemann metric $g$ to be defined by the matrix 
$\left(\begin{array}{cc}
                                 g^{11}&g^{12}\\
                                 g^{12}&g^{22}\\
                       \end{array}\right).$ 
It is clear that $T=\star_g$. Now the asymptotics $dTdf=O(exp)$ finishes the proof.
\end{proof}

\section{Proof of Theorem \ref{conj}.}
\label{peack}
The proof of the next proposition is long and we postpone in to Section \ref{technanal}.
\begin{predlozhenie}
Let $f\in \Lambda_m$ and assume that there
exists a locally defined Riemannian metric $g$ with such that 
$$\triangle_gf=O(exp),$$ then we can modify $g$ in an exponentially small fashion
to get that $$\triangle_gf=0$$
\label{final}
\end{predlozhenie}

Now the concatenation of Propositions \ref{approx} and \ref{final} gives us the following if and only if statement. 
\begin{teorema}
Let $f\in \Lambda_m$ and set $h:=[f]_{m+1\le 2m-4}$. Then there exists a locally defined metric $g$ which makes $f$ harmonic if and only if Equation \eqref{nonlineq}
$$\theta^{\star}\phi(h)=0$$ holds.
\label{central}
\end{teorema}

This allows us to complete the proof of Theorem \ref{conj} and provide a rigorous dressing for Theorem \ref{rigour}. 
For convenience of the reader we briefly recall the milestones: Proposition
\ref{harmreform} reduces the question of diffeomorphic equivalence to the existence
of a metric making a given function harmonic (``inverse problem for the Laplace operator''); Theorem \ref{central} reduces the problem to whether or not Equation \eqref{nonlineq} holds.

For part (1) of 
Theorem \ref{conj} consider $f\in \Lambda_m^{s(m)}$, that is $[f]_{s(m)}=f_0$. 
Then $h=0$, thus $\Theta^{\star}\phi(h)=0$ and we get a metric $g$ locally defined metric $g$ which makes $f$ harmonic near $(0,0)$. Then Proposition \ref{harmreform} gives us that $f$ is diffeomorphically equivalent to $f_0$.

For part (2) of Theorem \ref{conj} assume that $m\ge 5$ and consider 
$$f_{\star}=Re(x+iy)^m+C(x^2+y^2)^{m-2}\in \Lambda_m^{2m-5}\setminus \Lambda_m^{2m-4}$$
for $C\ne 0$. Propositions \ref{harmreform} and \ref{approx} tell us that it is
enough to check that $\Theta^{\star}_k\phi_k(h)\ne 0$ for some $k\in S$ 
with $h=C(x^2+y^2)^{m-2}$. We take $k=m-4$ for this purpose. We apply Remark \ref{comprhs} to the function $f_{\star}$ to get that
$$\phi_{m-4}(h)=-\triangle_{st}C(x^2+y^2)^{m-2}=-4C(m-2)^2(x^2+y^2)^{m-3}$$  
and recall Remark \ref{missed} to get that $\phi_{m-4}(h)\notin Im\Theta_{m-4}$, so
$$\Theta^{\star}_{m-4}\phi_{m-4}(h)\ne 0.$$ This completes the proof of Theorem \ref{conj}.

Now we give a rigorous way of saying that
``{\it functions from $\Lambda_m$ equivalent to $f_0$ form a submanifold of codimension
$(m-2)(m-3)-2$ in $\Lambda_m$}''. 
Since we have reduced the problem of diffeomorphic equivalence
to a certain equation in $\mathbb{R}_{m+1\le 2m-4}[x,y]$, it is natural to replace functions by their truncated Taylor power series in the italic sentence above. 
\begin{teorema}
Assume that $m\ge 5$. Then
the set $$E:=\{[f]_{m+1\le 2m-4}|f\in \Lambda_m,f \sim f_0\}$$
is a submanifold in $\mathbb{R}_{m+1\le 2m-4}[x,y]$ of codimension $(m-2)(m-3)-2$.
Moreover, for any $f\in \Lambda_m$ with $[f]_{m+1\le 2m-4}\in E$ 
we have that $f\sim f_0$.
\label{codim}
\end{teorema}
\begin{proof}
In view of Theorem \ref{central} for $f\in \Lambda_m$ to be equivalent to $f_0$
is the same that for $h=[f]_{m+1\le 2m-4}$ to satisfy Equation \eqref{nonlineq}.
Lemma \ref{submers} completes the proof.
\end{proof}
Theorem \ref{codim} is a rigorous version of Theorem \ref{rigour}.

\section{Technical analysis around zero.}
\label{technanal}
In this section we prove Proposition \ref{final}. We restate it here to fix some notation:

\begin{predlozhenie}
Let $f\in \Lambda_m$ and assume that there
exists a locally defined metric $\tilde g$ such that 
$$\triangle_gf=O(exp)$$ on some neighbourhood around $(0,0)$. Then there exists a Riemannian $g$, defined on a possibly smaller neighbourhood of $(0,0)$ than 
$\tilde g$ was such that $$\triangle_gf=0$$ and moreover $$g-\tilde g=O(exp).$$
\label{techfinal}
\end{predlozhenie}

We set up the machinery which 
starts with a $C^{\infty}$-metric making $f$ harmonic 
``up to order $l$'' at the origin
and produces a $C^l$-metric out of it making $f$ honestly harmonic,
$l=0,1,...,\infty$. Doing this for $l=\infty$ would obviously finish the job.
It turns out, however, that it is convenient to start out slowly with $l=0$,
postponing the case $l=\infty$ until later. Note that the most naive metric ---
the standard one in coordinates $(x,y)$ --- already makes $f$ harmonic 
``up to order 0''
at the origin. So for the next proposition we do not need any further assumptions 
on $f$.

\begin{predlozhenie}
Let $f\in \Lambda_m$. Then
there exists a continuous Riemannian metric $g$ on some open neighbourhood $U$
around $(0,0)\in \mathbb{R}^2$, which makes $f$ harmonic. 
Moreover the metric is smooth in the punctured open neighbourhood 
$U\setminus \{(0,0)\}$.  
\label{approxi}
\end{predlozhenie}
\begin{proof} Clearly,
$Kerf_{0x}\cap Kerf_{0y}=\{(0,0)\}$.  
Let $g$ be the desired Riemannian metric. Recall that with our the convention that $\det g=1$, the equation $\triangle_{g}f=0$ for $g$ reads as 
\begin{equation}
(g^{12}f_x+g^{22}f_y)_y+(g^{11}f_x+g^{12}f_y)_x=0.
\label{gperb}
\end{equation}
Assume for the moment, that Equation \eqref{gperb} is solved by a 
Riemannian metric $g$ with the regularity we want.
Then the combination 
\begin{equation}
A=g^{12}f_x+g^{22}f_y, 
\label{ej}
\end{equation}
gives us a function on
$\mathbb{R}^2$ with a-priori the same regularity as $g$ has. Analogously 
\begin{equation}
B=g^{11}f_x+g^{12} f_y. 
\label{bi}
\end{equation}
Note that $A_y+B_x=0$. 
Our convention about the determinant 
of $g$ writes out as 
\begin{equation}
g^{11}g^{22}-(g^{12})^2=1.
\label{det}
\end{equation}
The set of equations \eqref{ej}, \eqref{bi} and \eqref{det} can be viewed as 
a system of equations on our matrix elements $\{g^{ij}\}^{i,j=1,2}$. To solve
this system we express $g^{11}$ and $g^{22}$ in terms of $g^{12}$, $A$ and 
$B$ using \eqref{ej} and \eqref{bi}:
\begin{equation}
g^{11}=\frac{B-g^{12}f_y}{f_x},
\label{g11}
\end{equation}
\begin{equation}
g^{22}=\frac{A-g^{12}f_x}{f_y}.
\label{g22}
\end{equation}
and substitute these expressions in \eqref{det}:
$$\frac{B-g^{12}f_y}{f_x}\frac{A-g^{12}f_x}{f_y}-
(g^{12})^2=1, 
$$ and hence,
\begin{equation}
g^{12}=\frac{AB-f_x f_y}{Af_y+Bf_x}
\label{g12}
\end{equation}
Formulas $\eqref{g11}$, $\eqref{g22}$ and $\eqref{g12}$ can be viewed as an 
expression of our matrix elements $g^{11}$, $g^{22}$, $g^{12}$ through the 
functions $A$ and $B$. Of course while writing out these formulas 
we have divided
by zero in several places, but it does not hurt to do this at the moment, 
since our computations were done under the assumption that $g$
is well defined on the whole of $U$ a-priory. 

Now we change the 
direction of the logic. We want to solve Equation \eqref{gperb} together with
\eqref{det}, thus obtaining the desired Riemannian metric. For this we give
ourselves smooth functions $A$ and $B$, defined in some open neighbourhood
$U$ around the origin with 
$A_y+B_x=0$. The freedom of this choice will be exploited later. We insert 
the 
functions $A$ and $B$ in the system \eqref{ej}, \eqref{bi} and \eqref{det}
as a right hand side and note that the solution to this system will 
automatically satisfy \eqref{gperb} and \eqref{det}, thus giving the 
Riemannian
metric we want provided that regularity questions are taken care of. 
Unfortunately, the direct
usage of the formulas \eqref{g11}, \eqref{g22} and \eqref{g12} 
in order to solve the system \eqref{ej}, \eqref{bi}, \eqref{det} will run 
into 
problems like division by zero. Therefore, we will do the following. First, we
exploit the freedom in the choice of the functions $A$ and $B$ by fixing 
their
principal parts properly. The higher order terms remain arbitrary. 
Next, we will see, that Formula \eqref{g12} does not have problems in a small 
neighbourhood of the origin, and hence defines a function $g^{12}$ in
this small neighbourhood. The function $g^{11}$ will be defined in two steps.
First, we use Formula \eqref{g11} to defined it away from the set where
the corresponding denominator is small and then we use Equation \eqref{det}
to extend it over the problematic set. The function $g^{22}$ is defined
analogously. The last step is to show that the so defined functions 
$g^{12}$, $g^{11}$ and $g^{22}$ do satisfy the system 
\eqref{ej}, \eqref{bi}, \eqref{det} which is not automatic, because the 
formulas \eqref{g11} and \eqref{g22} do not apply everywhere 
in the domain of definition of the functions  $\{g^{ij}\}^{i,j=1,2}$.
 
Now we carry out this plan. We set $A_{m-1}=f_{0y}$, 
$B_{m-1}=f_{0x}$ to be the principal parts of 
$A=A_{m-1}+A_r$ and $B=B_{m-1}+B_r$ respectively, where
$A_r$ and $B_r$ are left to be arbitrary smooth functions  
of the order higher than $m-1$, subject to the relation $(A_r)_y+(B_r)_x=0$
(the freedom in the choice of $A_r$ and $B_r$ will be exploited in the proof of Proposition \ref{diffur1} when we improve the regularity of $g$ at the origin).
First, we analyze Formula \eqref{g12}: 
$$g^{12}=\frac{AB-f_xf_y}{Af_y+Bf_x}=
\frac{(A_{m-1}+A_r )(B_{m-1}+B_r )-(f_{0x}+\phi_1)(f_{0y}+\phi_2)}
{(A_{m-1}+A_r )(f_{0y}+\phi_2)+(B_{m-1}+B_r )(f_{0x}+\phi_1)}=$$ 
$$=\frac{A_{m-1}B_{m-1}-f_{0x}f_{0y}+r^{12}_n}
{A_{m-1}f_{0y}+B_{m-1}f_{0x}+r^{12}_d}=
\frac{r^{12}_n}{f_{0y}^2+f_{0x}^2+r^{12}_d},$$
where $$r^{12}_n=A_{m-1}B_r +A_r B_{m-1}-f_{0x}\phi_2-f_{0y}\phi_1+
A_r B_r -\phi_1\phi_2  $$ and
$$r^{12}_d=A_{m-1}\phi_2+A_r f_{0y}+B_{m-1}\phi_1+B_r f_{0x}+A_r \phi_2+
B_r \phi_1,$$ where $\phi_1$ and $\phi_2$ 
are partial derivatives of $f-f_0$ with respect to $x$ and $y$ respectively.
Note that $f_{0x}^2+f_{0y}^2$ is a nowhere zero homogeneous polynomial of 
order $2(m-1)$, so the estimate \eqref{poly1} applied to 
$f_{0x}^2+f_{0y}^2$
implies that 
$$\frac{r^{12}_n}{f_{0y}^2+f_{0x}^2}=o(1)$$ at 
$(0,0)$ 
and $$\frac{r^{12}_d}{f_{0y}^2+f_{0x}^2}=o(1)$$ at $(0,0)$. Therefore, the 
function 
$g^{12}$ is well-defined in some neighbourhood $U$ of the origin, belongs 
to the class $C^{\infty}(U \setminus (0,0))\cap C^{0}(U)$ and 
the relation $$\lim_{(x,y)\to (0,0)}g^{12}=g^{12}|_{(x,y)=(0,0)}=0$$ 
holds true. 
Now, we analyze the formulas \eqref{g11} and \eqref{g22} and 
define the functions $g^{11}$ and $g^{22}$. The idea is that 
for each formula we cut out ``problematic'' sectors and work
on those parts of $\mathbb{R}^2$ where we are guaranteed from small or
vanishing denominators. 
Since $f_{0x}=mRe(x+iy)^{m-1}$ is a homogeneous polynomial of order $m-1$,
not identically zero, we fix a small positive $\delta$ and set 
$\Omega_{11}=\Omega^{\delta}(f_{0x})$.
Next, we rewrite \eqref{g11} in a more convenient way:
$$g^{11}=\frac{B-g^{12}f_y}{f_x}=
\frac{B_{m-1}+B_r-g^{12}(f_{0y}+\phi_2)}{f_{0x}+\phi_1}=
\frac{f_{0x}+r^{11}_n}{f_{0x}+\phi_1}=\frac{f_{0x}(1+\frac{r^{11}_n}{f_{0x}})}
{f_{0x}(1+\frac{\phi_1}{f_{0x}})},$$ 
where $r^{11}_n=B_r-g^{12}f_{0y}-g^{12}\phi_2.$
To take a more precise look at Formula \eqref{g11} we restrict ourselves to 
$U\cap \Omega_{11}$. 
Now  Estimate \eqref{poly2} applied to $f_{0x}$ implies that 
$$\frac{r^{11}_n}{f_{0x}}|_{\Omega_{11}}=o(1)$$ at $(0,0)$
and $$\frac{\phi_1}{f_{0x}}|_{\Omega_{11}}=o(1)$$ at $(0,0)$.
 Therefore, $f_x|_{U\cap\Omega_{11}}$ has an isolated zero at the origin,
and the right hand side of \eqref{g11} is well-defined on 
$U\cap \Omega_{11}$ (we shrink the neighbourhood $U$ if necessary). At this 
point we set the function $g^{11}$ to be defined on $U\cap \Omega_{11}$ 
by Formula \eqref{g11}.   
The so defined function $g^{11}$ (only on $U\cap \Omega_{11}$ so far) 
exhibits the following regularity: 
$$g^{11}|_{U \cap \Omega_{11}} 
\in C^{\infty}((U\cap \Omega_{11})\setminus (0,0))
\cap C^0(U \cap \Omega_{11})$$
and the relation 
$$\lim_{(x,y)\to (0,0)}g^{11}|_{U\cap \Omega_{11}}=g^{11}|_{(x,y)=(0,0)}=1$$ 
holds true. 
The latter allows us to assume (by shrinking $U$ further if necessary)
that $g^{11}|_{U\cap \Omega_{11}}$ is nowhere zero. 
Similar discussions apply to Formula \eqref{g22}. In brief,
$\Omega_{22}=\Omega^{\delta}(f_{0y})$,
$$g^{22}=\frac{f_{0y}+r^{22}_n}{f_{0y}+\phi_2},$$ for 
$r^{22}_n$ being a function with the faster decay at $(0,0)$ than 
$(x^2+y^2)^{(m-1)/2}$. 
By the same token as before,  
$$\frac{r^{22}_n}{f_{0y}}|_{\Omega_{22}}=o(1)$$ and 
$$\frac{\phi_2}{f_{0y}}|_{\Omega_{22}}=o(1)$$ at $(0,0)$.
Therefore, the function $g^{22}$ is well-defined on $U \cap \Omega_{22}$
(the neighbourhood $U$ can be shrunk further if needed). 
Moreover, we have 
$$g^{22}|_{U\cap \Omega_{22}} 
\in C^{\infty}((U\cap \Omega_{22})\setminus (0,0))
\cap C^0(U \cap \Omega_{22})$$ and the relation 
$$\lim_{(x,y)\to (0,0)}g^{22}|_{U\cap \Omega_{22}}=g^{22}|_{(x,y)=(0,0)}=1$$ 
holds true. Therefore 
$g^{22}|_{U\cap \Omega_{22}}$ is nowhere zero. Since we know that 
$$Ker(f_{0x}) \cap Ker(f_{0y}) =\{(0,0)\},$$ 
we can choose
$\delta$ small enough and achieve that   
$$\Omega_{11}\cup \Omega_{22}=\mathbb{R}^2$$ and
$$Int(\Omega_{11}\cap \Omega_{22})\ne \emptyset.$$

Now comes a crucial moment. We are to 
extend the functions $g^{11}$ and $g^{22}$ to 
the whole of $U$.
Equation \eqref{det} 
holds true on the triple 
intersection $U\cap \Omega_{11}\cap \Omega_{22}$ by the formulas \eqref{g11}
and \eqref{g22}. Since  
$g^{22}|_{U \cap \Omega_{22}}$ is 
nowhere zero this equation equivalently reads as 
\begin{equation}
g^{11}=\frac{1+(g^{12})^2}{g^{22}}.
\label{g112} 
\end{equation}
The right hand side of this equation makes perfect sense and has 
the regularity required for the function $g^{11}$ on 
$U\cap \Omega_{22}$. This allows us to define the function
$g^{11}$ on $U\cap \Omega_{22}$ by \eqref{g112}.
So now we have defined the function $g^{11}$ on $U\cap \Omega_{11}$
via \eqref{g11} and on $U\cap \Omega_{22}$ via \eqref{g112}. The two
definitions overlap on $U\cap \Omega_{11}\cap \Omega_{22}$ and clearly agree
there, since Equation \eqref{det}, where the second definition has come from,
holds true on $U\cap \Omega_{11}\cap \Omega_{22}$ with $g^{11}$ defined
in the first way. Altogether, we have that the function $g^{11}$
is defined and has the regularity we need on both $U\cap \Omega_{11}$ and
$U\cap \Omega_{22}$ and hence on $U$ --- their union. Note that Equation 
\eqref{det}, after we have made this extension, holds true not only on 
$U\cap \Omega_{11}\cap \Omega_{22}$, but on the large set 
$U\cap \Omega_{22}$.
Analogously we extend the function $g^{22}$ from $U\cap \Omega_{22}$
to the whole of $U$. 

We remark that Equation \eqref{det} now holds true not
only on $U\cap \Omega_{22}$, but on the whole of $U$. 
Now we have come to the last step, i.e. we are to 
show that the so defined functions $g^{12}$, $g^{11}$ and $g^{22}$ do 
actually
satisfy the system \eqref{ej}, \eqref{bi}, \eqref{det} and hence both  
\eqref{gperb} and \eqref{det}, therefore, giving us the Riemannian metric 
$g$ 
which makes $f$ harmonic and has the 
$C^{\infty}(U \setminus (0,0))\cap C^0(U)$ regularity. \par
Equation \eqref{det} is satisfied automatically by the remark above. For 
\eqref{bi} we start we a point $(x,y)\in U$ and consider $g^{11}f_x$ 
at this point. Here we distinguish between the 
following two cases:\\ 1) $(x,y)\in \Omega_{11}$ and \\ 
2) $(x,y)\in \Omega_{22}$.\\
In the first case we are done by Formula \eqref{g11}. In the second
case Formula \eqref{g11} does not apply, but fortunately \eqref{g22} does
apply. For this we carry out an easy computation: 
$$g^{11}f_x|_{(x,y)}=
g^{11}g^{22}f_x f_y \frac{1}{g^{22}f_y}|_{(x,y)}=
(1+(g^{12})^2)f_x f_y \frac{1}{g^{22}f_y}|_{(x,y)}=$$ $$=
(AB-g^{12}(Af_y+Bf_x)+(g^{12})^2f_x f_y)
\frac{1}{g^{22}f_y}|_{(x,y)}=$$ $$=
\frac{(B-g^{12}f_y)(A-g^{12}f_x)}{g^{22}f_y}|_{(x,y)}=
(B-g^{12}f_y)|_{(x,y)}.$$ The first equality sign is valid, because 
$g^{22}$ is nowhere zero and 
$f_y|_{U\cap \Omega_{22}}$ has a unique zero at the origin. 
The second one is valid by \eqref{det}. 
The third one easily follows from the definition of $g^{12}$.
The fourth one is just an elementary algebra. The fifth one follows
from \eqref{g22}. This shows \eqref{bi}. It can be shown completely 
analogously that Equation \eqref{ej} is also satisfied. \end{proof}

Next proposition is the final step. We take up the case $l=\infty$. That is 
we find a smooth Riemannian metric $g$ making $f$ harmonic with 
$g-\tilde g=O(exp)$.
  
\begin{predlozhenie} Under the assumptions of Proposition \ref{techfinal}
the regularity of the Riemannian metric $g$ in Proposition \ref{approxi}
can be improved to $C^{\infty}$, moreover, we can achieve that 
$$g-\tilde g=O(exp).$$
\label{diffur1}
\end{predlozhenie}
\begin{proof}  
To take care of the regularity of $g$ constructed in Proposition \ref{approxi}
at $(0,0)$ we exploit the freedom in the choice of functions $A$ and $B$. 
Set $$\kappa:=\triangle_{\tilde g}f=
({\tilde g}^{12}f_x+{\tilde g}^{22}f_y)_y+
({\tilde g}^{11}f_x+{\tilde g}^{12}f_y)_x.$$
We shrink $U$ if necessary to make it convex and squeeze it into the domain of definition of $\tilde g$.
We introduce a smooth function $\xi$, defined on $U$, by the formula:
$$\xi(x,y)=\int_0^x\kappa(\tilde x,y)d\tilde x.$$ Clearly, the function $\xi$
is of the class $O(exp)$. 
We set $$A_{\infty}={\tilde g}^{12}f_x+{\tilde g}^{22}f_y,$$
$$B_{\infty}={\tilde g}^{11}f_x+{\tilde g}^{12}f_y$$ and 
pick some function $\phi\in O(exp)$ arbitrarily. 
Set $A=A_{\infty}+\phi_x$ and $B=B_{\infty}-\xi-\phi_y$.
First, we check, that the principal parts $A_{m-1}$ and $B_{m-1}$ of 
$A$ and $B$
coincide with those chosen in the proof of Proposition \ref{approxi}. Indeed,
$$A_{m-1}={\tilde g}^{12}_{(0,0)}f_{0x}+{\tilde g}^{22}_{(0,0)}f_{0y}=f_{0y}$$ 
and
$$B_{m-1}={\tilde g}^{11}_{(0,0)}f_{0x}+{\tilde g}^{12}_{(0,0)}f_{0y}=f_{0x}$$
as before. 
Next, we check, that $A$ and $B$ satisfy the condition 
$A_y+B_x=0$. Indeed, $$A_y+B_x=A_{\infty y}+\phi_{xy}+
B_{\infty x}-\phi_{yx}-\xi_x=\kappa-\kappa+\phi_{xy}-\phi_{yx}=0.$$ 
Now we analyze 
formula \eqref{g12} for the off diagonal element of the metric deeper
than previously. Basically, it follows the same pattern as before,
but now we want infinite differentiability of $g^{12}$ 
at the origin instead of just continuity.
$$g^{12}=\frac{AB-f_xf_y}{Af_y+Bf_x}=
\frac{(A_{\infty}+\phi_x)(B_{\infty}-\xi-\phi_y)-f_xf_y}
{(A_{\infty}+\phi_x)f_y+(B_{\infty}-\xi-\phi_y)f_x}=$$ $$=
\frac{A_{\infty}B_{\infty}-f_xf_y+r^{12}_n}
{A_{\infty}f_y+B_{\infty}f_x+r^{12}_d}=
\frac{\frac{A_{\infty}B_{\infty}-f_xf_y}{A_{\infty}f_y+B_{\infty}f_x}+
\tilde r^{12}_n}{1+\tilde r^{12}_d}=
\frac{A_{\infty}B_{\infty}-f_xf_y}{A_{\infty}f_y+B_{\infty}f_x}+r,$$
where the functions $r^{12}_n$, $r^{12}_d$, $\tilde r^{12}_n$, 
$\tilde r^{12}_d$, $r$ are all of the class $O(exp)$ and the neighbourhood
$U$ around the origin we are working at is taken to be small enough 
for $A_{\infty}f_y+B_{\infty}f_x$
to be nonzero in the punctured neighbourhood.
By the choice of $A_{\infty}$ and $B_{\infty}$, we have that  
$$\frac{A_{\infty}B_{\infty}-f_xf_y}{A_{\infty}f_y+B_{\infty}f_x}=
{\tilde g}^{12},$$
therefore $g^{12}={\tilde g}^{12}+r$, in particular $g^{12}$ is smooth. To 
work out the desired regularity for diagonal elements is a little harder.
First, we consider the difference $g^{11}-{\tilde g}^{11}$ restricted to the 
set
$U\cap \Omega_{11}$, where the formula \eqref{g11} works: 
$$(g^{11}-{\tilde g}^{11})|_{U\cap\Omega_{11}}=\frac{B-g^{12}f_y}{f_x}-
{\tilde g}^{11}=
\frac{B_{\infty}-\xi-\phi_y-({\tilde g}^{12}+r)f_y}{f_x}-{\tilde g}^{11}=$$ $$=
\frac{B_{\infty}-{\tilde g}^{12}f_y+r^{11}_n}{f_x}=
{\tilde g}^{11}+\tilde r^{11}-{\tilde g}^{11}=
\tilde r^{11},$$ where the function $r^{11}_n$ is of the class $O(exp)$
and we have to be a little more careful about the function $\tilde r^{11}$.
It is smooth on the set $U\cap \Omega_{11}\setminus \{(0,0)\}$ and decays 
at $(0,0)$ together with all its derivatives faster than any polynomial.
Analogously, $g^{22}|_{U\cap \Omega_{22}}={\tilde g}^{22}+\tilde r^{22}$, where the 
function $\tilde r^{22}$ is smooth on the set 
$U \cap \Omega_{22}\setminus \{(0,0)\}$ and decays at $(0,0)$ 
together with all its
derivatives faster than any polynomial. This allows us to write out
the difference $g^{11}-{\tilde g}^{11}$ restricted to the set
$U\cap \Omega_{22}$: 
$$(g^{11}-{\tilde g}^{11})|_{U \cap \Omega_{22}}=\frac{1+(g^{12})^2}{g^{22}}-
{\tilde g}^{11}=\frac{1+({\tilde g}^{12})^2+2{\tilde g}^{12}r+r^2}
{{\tilde g}^{22}+\tilde r^{22}}-{\tilde g}^{11}=$$ $$=
\frac{1+({\tilde g}^{12})^2}{{\tilde g}^{22}}+\hat r^{11}-{\tilde g}^{11}={\tilde g}^{11}+\hat r^{11}-{\tilde g}^{11}=
\hat r^{11},$$ where the function $\hat r^{11}$
is smooth on the set $(U \cap \Omega_{22})\setminus \{(0,0)\}$ and decays at 
$(0,0)$ together with all its derivatives faster than any polynomial.
Altogether, we have that the difference $g^{11}-{\tilde g}^{11}$ is smooth in 
a punctured neighbourhood of $(0,0)$ and decays at $(0,0)$ together with 
all its derivatives faster than any polynomial (in the above calculations
we shrink the neighbourhood $U$ of the origin whenever necessary, to 
keep track of the denominators). Consequently, the difference
$g^{11}-{\tilde g}^{11}$ is of the class $O(exp)$. In particular, the upper 
left element $g^{11}$ of the metric is smooth. The lower right element $g^{22}$
can be treated analogously. The above calculations show that the metric $g$ 
is smooth and moreover, $g-\tilde g=O(exp)$. \end{proof}

\end{document}